\documentclass[12pt,a4paper]{amsart}
\usepackage[width=0.75\paperwidth,height=0.85\paperheight,twoside=false,includehead,includefoot]{geometry}
\allowdisplaybreaks[1]

\DeclareFontEncoding{OT2}{}{}
\DeclareFontSubstitution{OT2}{cmr}{m}{n}

\DeclareFontFamily{OT2}{cmr}{\hyphenchar\font45 }
\DeclareFontShape{OT2}{cmr}{m}{n}{%
   <5><6><7><8><9>gen*wncyr%
   <10><10.95><12><14.4><17.28><20.74><24.88>wncyr10}{}

\DeclareMathAlphabet{\mathcyr}{OT2}{cmr}{m}{n}
\DeclareMathAlphabet{\mathcyb}{OT2}{cmr}{b}{n}
\SetMathAlphabet{\mathcyr}{bold}{OT2}{cmr}{b}{n}

\usepackage{amsopn}

\DeclareMathOperator*{\Max}{Max}

\usepackage{amstext}
\usepackage{amsthm}
\usepackage{amssymb}

\newtheorem{thm}{Theorem}[section]
\newtheorem{lem}[thm]{Lemma}
\newtheorem{prop}[thm]{Proposition}
\newtheorem{cor}[thm]{Corollary}

\theoremstyle{definition}

\newtheorem{ex}[thm]{Example}
\theoremstyle{remark}
\newtheorem*{rem}{Remark}

\begin{document}

\title[On a basis for Euler-Zagier double zeta functions]{On a basis for Euler-Zagier double zeta functions with non-positive components}

\author{Hideki Murahara}
\address[Hideki Murahara]{Nakamura Gakuen University Graduate School, 5-7-1, Befu, Jonan-ku,
Fukuoka, 814-0198, Japan}
\email{hmurahara@nakamura-u.ac.jp}

\author{Takashi Nakamura}
\address[Takashi Nakamura]{Department of Liberal Arts, Faculty of Science and Technology, Tokyo University of Science, 2641 Yamazaki, Noda-shi, Chiba-ken, 278-8510, Japan}
\email{nakamuratakashi@rs.tus.ac.jp}


\begin{abstract}
 For a non-negative integer $N$, let 
 $\mathcal{Z}_{N}:=\sum^N_{c = 0} \mathbb{Q} \cdot \zeta(-c,s+c)$, 
 where the right-hand side is the vector space spanned by the Euler-Zagier double zeta functions over $\mathbb{Q}$. 
 In this paper, we show that 
 $\mathcal{Z}_{N}
  =\bigoplus^{N}_{c = 0 : \text{even}} \mathbb{Q} \cdot \zeta(-c,s+c)$, 
 where $\bigoplus$ is the direct sum of vector spaces. 
 Moreover, we give a family of relations that exhaust all $\mathbb{Q}$-linear relations on $\mathcal{Z}_{N}$.
\end{abstract}

\maketitle

\tableofcontents

\section{Introduction}
The Euler-Zagier multiple zeta function (MZF) is defined by the series
\begin{align*}
 \zeta(s_{1},\dots,s_{r}):=\sum_{1\le n_{1}<\cdots<n_{r}}\frac{1}{n_{1}^{s_{1}}\cdots n_{r}^{s_{r}}}
\end{align*}
for complex valuables $s_1,\ldots,s_r \in\mathbb{C}$, which is convergent absolutely in the domain 
\begin{align} \label{abc}
 \{(s_{1},\ldots,s_{r})\in\mathbb{C}^{r}\;|\;\Re(s_r)>1,\dots,\Re(s_1+\cdots+s_r)>r \}
\end{align}
(see \cite{Mat02}). 
It is known that the function $\zeta(s_{1},\dots,s_{r})$ can be meromorphically continued to the whole space $\mathbb{C}^{r}$
(see \cite{AET01} and \cite{Zha00}). 

The special values of MZF at positive integer points are called multiple zeta values (MZVs). 
The MZVs have emerged in various areas of mathematics and physics, inspiring interest among researchers 
(see (\cite{Zag94,LM95,LM96}), for example). 
In particular, the algebraic structure of MZV has been widely studied, and many kinds of algebraic relations over $\mathbb{Q}$ are known. 
The algebraic relations such as Extended double shuffle relation (\cite{IKZ06}), associator relation (\cite{Fur11}), confluence relation (\cite{HS19}), and Kawashima's relation (\cite{Kaw09}) are expected to exhaust all relations of MZVs, respectively. 
Note that some relations among MZVs are generalized to MZFs (see \cite{HMO18,HMO19,MO20}, for example).

For MZVs, the existence of a basis such as the Hoffman basis is also known. 
The Hoffman basis is the set of MZVs $\zeta(k_1,\dots,k_r)$ with $k_1,\dots,k_r\in\{2, 3\}$ and every MZV can be expressed as a $\mathbb{Q}$-linear combination of this set. 
This fact was conjectured by Hoffman \cite{Hof97} and proved by Brown \cite{Bro12}. 
Attempts have been made to understand the dimensions of weight graded spaces of MZVs (here, the weight means the sum of every component in MZVs, i.e., $k_1+\cdots+k_r$ for $\zeta(k_1,\dots,k_r)$). 
Zagier \cite{Zag94} gave the dimension conjecture, and Deligne and Goncharov \cite{DG05} and Terasoma \cite{Ter02} proved the upper bound of this conjecture. 

It should be noted that the Hoffman basis is not a basis in the ordinary sense, namely, it is not even known that $\zeta (2,3) / \zeta (3,2)$ is rational or not. 
For the same reason, it seems to be almost impossible to give the lower bound of the dimension conjectured by Zagier or show that the relations above give all $\mathbb{Q}$-linear relations of MZVs.

\section{Main Theorems}
While much research has been done on relations of MZFs/MZVs, and the basis and the dimensions of MZVs, studies on MZFs other than MZVs in the absolute convergence region has not been done so far. 
For a non-negative integer $N$, let 
\begin{align*}
 \mathcal{Z}_{N}
 &:=\sum_{0\le c\le N} \mathbb{Q}\cdot \zeta(-c,s+c), \\
 \mathcal{Z}
 &:=\lim_{N\rightarrow\infty} \mathcal{Z}_{N},
\end{align*}
where the right-hand side of the first equation is the vector space spanned by the Euler-Zagier double zeta functions over $\mathbb{Q}$. 
In this paper, we give a family of relations that exhaust all $\mathbb{Q}$-linear relations, an explicit basis, and the dimension of $\mathcal{Z}_{N}$ and $\mathcal{Z}$.
It should be emphasized that the basis given in Theorem \ref{th:main2} is a basis in the normal sense, in other words, every element of $\mathcal{Z}$ can be written in a unique way as a linear combination of elements of this basis. 
Furthermore, Theorem \ref{th:main2} enables us to show Theorem 2.1 which implies that the relations in \eqref{eq:eq1} give all $\mathbb{Q}$-linear relations of $\mathcal{Z}_N$. 
These theorems are the most noteworthy results in this paper since there is no such fact on ordinary MZVs as mentioned at the end of Section 1.

Now we fix a positive integer $N$ and put $N':=\lfloor N/2 \rfloor$ where $\lfloor x\rfloor$ is the floor function defined by $\lfloor x\rfloor :=\Max\{ n\in \mathbb{Z} \mid n\le x \}$.
Let $A$ be the $2N' \times N'$ matrix whose components $\{a_{c,d}\}$ are defined by the following rules:
 \begin{itemize}
  \item $a_{c,1}=1 \quad  (c\ge1)$, 
  \item $a_{1,d}=0,a_{2,d}=0 \quad  (d\ge2)$,  
  \item $a_{3,2}=-2$, $a_{3,d}=0 \quad  (d\ge3)$,
  \item $a_{c+2,d+1}=a_{c+1,d+1}-a_{c,d} \quad (c\ge2,d\ge1)$.  
 \end{itemize} 
Note that this matrix can be also defined by Lemma \ref{lem:acd}. 
Let $A_1$ and $A_2$ be 
$N' \times N'$ submatrices of $A$ whose $(i,j)$ components $a^{(1)}_{i,j}$ and $a^{(2)}_{i,j}$ are 
defined by
\begin{align*}
 a^{(1)}_{i,j}
 =a_{2i-1,j},
 \qquad 
 a^{(2)}_{i,j}
 =a_{2i,j}.
\end{align*}
Note that the matrices $A_1$ and $A_2$ are lower triangular matrices in which all the elements of the diagonal are non-zero, 
and have inverse matrices (see also Sections 3 and 6).
Put
\begin{align*}
 \boldsymbol{z}_1
 &:=(\zeta(0,s)/2,\zeta(-2,s+2),\zeta(-4,s+4),\dots, \zeta(-2N'+2,s+2N'-2) )^\mathsf{T}, \\
 \boldsymbol{z}_2
 &:=(\zeta(-1,s+1),\zeta(-3,s+3),\dots, \zeta(-2N'+1,s+2N'-1) )^\mathsf{T}.
\end{align*} 
Then our first theorem is as follows. 
\begin{thm} \label{th:main1}
 Let $s\in\mathbb{C}$. 
 Except for the singularities, we have
 \begin{align}
  \label{eq:eq1}
  A_1^{-1}\,\boldsymbol{z}_1 -A_2^{-1}\,\boldsymbol{z}_2&=\boldsymbol{0}. 
 \end{align} 
 Moreover, all $\mathbb{Q}$-linear relations of $\mathcal{Z}_{N}$ are deduced from this relation.
\end{thm}
\begin{rem}
 The singularities of $\zeta (-k,s+k)\; (0 \le k \le N)$ are given in Theorem \ref{th:pole}. 
 We give some specific examples of $\mathbb{Q}$-linear relations in Section 6.
\end{rem}

As a consequence of Theorem \ref{th:main1}, we give bases of the spaces $\mathcal{Z}_{N}$ and $\mathcal{Z}$.
\begin{thm} \label{th:main2}
 For a non-negative integer $N$, we have
 \begin{align}
  \label{eq:base1}
  \mathcal{Z}_{N}
  &=\bigoplus_{\substack{ c\text{:even} \\ 0\le c\le N}}
   \mathbb{Q}\cdot \zeta(-c,s+c), \\
  \label{eq:base2}
  \mathcal{Z}
  &=\bigoplus_{\substack{ c\text{:even} \\ c\ge0 }}
   \mathbb{Q}\cdot \zeta(-c,s+c). 
 \end{align}
 Moreover, we have 
 \begin{align} \label{eq:dim}
  \dim_{\mathbb{Q}} \mathcal{Z}_{N} 
  =\left\lfloor \frac{N}{2} \right\rfloor +1.
 \end{align} 
\end{thm}

The rest of our paper is organized as follows. 
In Section 3, we prove \eqref{eq:eq1} in Theorem \ref{th:main1} which asserts 
$\dim_{\mathbb{Q}} \mathcal{Z}_{N} \le \lfloor N/2 \rfloor + 1$ 
by stating the relationship between the Euler-Zagier double zeta function and the Tornheim double zeta function.
In Section 4, we show some analytic properties of $\zeta (-n,s+n)$ which imply $\dim_{\mathbb{Q}} \mathcal{Z}_{N} \ge \lfloor N/2 \rfloor + 1$ and will be used in Section 5. 
Our main theorem is proved in Section 5 by the facts given in Sections 3 and 4. 
Section 6 gives concrete examples.

\section{$\mathbb{Q}$-linear relations of the spaces $\mathcal{Z}_{N}$ and $\mathcal{Z}$}
In this section, we prove Theorem \ref{th:main1}. 
We prove the following equality, first. 
\begin{lem} \label{lem:main1}
 Let $2, 1 \ne s \in {\mathbb{C}}$. 
 Then we have \begin{align*}
  \zeta (0,s) 
  =2 \zeta (-1,s+1).
 \end{align*}
\end{lem}
\begin{proof}
 By the definition of the Euler-Zagier double zeta function, we have
 \begin{align*}
  2 \zeta (-1,s+1) 
  &=2 \sum_{m,n=1}^\infty 
   \frac{1}{ m^{-1}(m+n)^{s+1} } \\
  &=\sum_{m,n=1}^\infty \frac{m}{ (m+n)^{s+1} }
   +\sum_{m,n=1}^\infty \frac{n}{ (m+n)^{s+1} } \\
  &=\sum_{m,n=1}^\infty \frac{m+n}{ (m+n)^{s+1} } \\
  &=\zeta (0,s) 
 \end{align*}
 when $\Re(s)>2$. 
 By the analytic continuation given in Theorem \ref{th:pole}, we have the lemma.    
\end{proof}

To prove Theorem \ref{th:main1}, we need Lemma \ref{lem:acd} and Propositions \ref{prop:Tornheim} and \ref{prop:lowerMat}. 
Recall that we fixed an integer $N$ and put $N':=\lfloor N/2 \rfloor$ in Section 2.
For non-negative integers $a$ and $b$, let $f(a,b):=m^a n^a (m+n)^b$. 
\begin{lem} \label{lem:acd}
 For a positive integer $c$, we have
 \begin{align*}
  m^c+n^c=\sum_{d=1}^{N'} a_{c,d} f(d-1,c-2d+2). 
 \end{align*}
\end{lem}
\begin{proof}
 We prove this lemma by induction on $c$. 
 The cases $c=1,2$ are ovbiously hold. 
 
 From the equality
 \[
  m^{c+1}+n^{c+1}
  =(m+n)(m^{c}+n^{c}) -mn (m^{c-1}+n^{c-1}), 
 \]
 and the induction hypothesis, we have
 \begin{align*}
  &m^{c+1}+n^{c+1} \\
  &=(m+n) \sum_{d=1}^{N'} a_{c,d} f(d-1,c-2d+2)
   -mn \sum_{d=1}^{N'} a_{c-1,d} f(d-1,c-2d+1) \\
  &=\sum_{d=1}^{N'} a_{c,d} f(d-1,c-2d+3)
   -\sum_{d=1}^{N'} a_{c-1,d} f(d,c-2d+1). 
 \end{align*}
 By the definition of the double sequence $a_{c,d}$ given in Section 2, we find
 \begin{align*}
  &m^{c+1}+n^{c+1} \\   
  &=a_{c+1,1} f(0,c+1) +\sum_{d\ge2} (a_{c,d} -a_{c-1,d-1}) f(d-1,c-2d+3) \\ 
  &=\sum_{d=1}^{N'} a_{c+1,d} f(d-1,c-2d+3). 
 \end{align*}
 This finishes the proof. 
\end{proof}

Let $T(s_1,s_2;s)$ be the Tornheim double zeta function (see \cite{Nak06}, for example), i.e.,
\[
 T(s_1,s_2;s)
 := \sum_{m,n=1}^\infty \frac{1}{m^{s_1} n^{s_2} (m+n)^s}.
\]
\begin{prop} \label{prop:Tornheim} 
 For a positive integer $c$ and $s\in\mathbb{C}$ with $\sigma>2$, we have
 \begin{align*}
  \zeta(-c,s+c)
  =\sum_{d=1}^{N'} \frac{a_{c,d}}{2} \, T(-d+1,-d+1;s+2d-2).  
 \end{align*}
\end{prop}
\begin{proof}
 By Lemma \ref{lem:acd}, have
 \begin{align*}
  2\zeta(-c,s+c)
  &=\sum_{m,n\ge1} \frac{m^c+n^c}{(m+n)^{s+c}} \\
  &=\sum_{d=1}^{N'} \sum_{m,n\ge1} \frac{a_{c,d} f(d-1,c-2d+2)}{(m+n)^{s+c}} \\
  &=\sum_{d=1}^{N'} \sum_{m,n\ge1} \frac{a_{c,d} m^{d-1}n^{d-1} }{(m+n)^{s+2d-2}} \\
  &=\sum_{d=1}^{N'} a_{c,d} T(-d+1,-d+1,s+2d-2). \qedhere
 \end{align*}
\end{proof}

The following proposition is prove by the properties of cofactor matrix, inverse matrix, and triangular matrix. 
\begin{prop}[{\cite[Theorem 2]{DB14}}] \label{prop:lowerMat}
 Let $A:=(a_{i,j})$ be a $m\times m$ lower triangular matrix whose diagonal components are not $0$. 
 Then the matrix $A$ has the inverse matrix. 
 Furthermore, we have 
 \begin{align*}
  A^{-1}
  =\left(
 \begin{array}{cccccc}
  \frac{1}{a_{1,1}} &0&\cdots&\cdots&0  \\
  a'_{2,1} & \frac{1}{a_{2,2}} &0&\cdots&0 \\
  \vdots &&\ddots &\ddots&\vdots \\  
  \vdots &&&\ddots&0 \\  
  a'_{m,1} & \cdots &\cdots& a'_{m,m-1} &\frac{1}{a_{m,m}}  \\
 \end{array}
 \right), 
 \end{align*}
 where
 \begin{align*}
  a'_{i,j}=(-1)^{i-j} \frac{D_{i,j}}{a_{j,j} \cdots a_{i,i}}
  \qquad (i>j). 
 \end{align*}
 Here we set
 \begin{align*}
  D_{i,j}:=
  \begin{cases}
   \quad a_{i,j} &\textrm{if }\,\, i=j+1, \\
   \left|
   \begin{array}{cccccccc}
    a_{j+1,j} & a_{j+1,j+1} &0&0&\cdots&0 \\
    a_{j+2,j} & a_{j+2,j+1} & a_{j+2,j+1} &0&\cdots&0 \\
    \vdots &\vdots &\vdots& \ddots &&\vdots \\  
    a_{i-2,j} & a_{i-2,j+1} & \cdots & \cdots & a_{i-2,i-1} &0 \\
    a_{i-1,j} & a_{i-1,j+1} & \cdots & \cdots & \cdots & a_{i-1,i-1} \\
    a_{i,j} & a_{i,j+1} & \cdots & \cdots & \cdots & a_{i,i-1} \\
 \end{array}
 \right|
  &\textrm{if }\,\, i>j+1.
  \end{cases}
 \end{align*}
\end{prop}

Now we prove the equality \eqref{eq:eq1} in Theorem \ref{th:main1}.
\begin{proof}[Proof of the equality \eqref{eq:eq1}]
In Lemma \ref{lem:main1}, we have already proved the first component of both vertical vectors in \eqref{eq:eq1}. 
Thus we prove the remainder here. 
Putting 
\begin{align*}
 \boldsymbol{t}
 &:=\biggl( 
  T(0,0;s),T(-1,-1;s+2),\dots,T\left(-N'+1,-N'+1;s+2N'-2 \right) 
  \biggr)^\mathsf{T},  
\end{align*} 
we have $\boldsymbol{z}=(A_1/2)\boldsymbol{t}$ and $\boldsymbol{z}=(A_2/2)\boldsymbol{t}$ by Proposition \ref{prop:Tornheim}.
Since the matrices $A_1$ and $A_2$ are lower triangular matrices in which all the elements of the diagonal are non-zero, we obtain the result by Proposition \ref{prop:lowerMat} (see Example 6.2, the matrices $A_1$ and $A_2$ with $N=12$ are given). 
\end{proof}

At the end of this section, we show Corollary \ref{cor:sumis0}. 
To prove this, we use the following fact. 
\begin{lem} \label{lem:simplefact}
 We have
 \begin{align*}
  \frac{
  \left|
   \begin{array}{cccccccc}
    1 & a_{2,2} &0&0&\cdots&0 \\
    1 & a_{3,2} & a_{3,3} &0&\cdots&0 \\
    \vdots &\vdots &\vdots& \ddots & \ddots&\vdots \\  
    \vdots &\vdots &\vdots&& \ddots & 0 \\  
    1 & a_{i-1,2} & \cdots & \cdots & \cdots & a_{i-1,i-1}  \\
    1 & a_{i,2} & \cdots & \cdots & \cdots & a_{i,i-1}  \\
   \end{array}
  \right|}
  {a_{2,2}\cdots a_{i,i}}
  =
  \sum_{j=2}^{i} (-1)^{j}
  \frac{
  \left|
   \begin{array}{cccccccc}
    a_{j+1,j} & a_{j+1,j+1} &0&\cdots&0 \\
    \vdots &\vdots& \ddots & \ddots&\vdots \\  
    \vdots &\vdots&& \ddots & 0 \\  
    a_{i-1,j} & \cdots & \cdots & \cdots & a_{i-1,i-1}  \\
    a_{i,j} & \cdots & \cdots & \cdots & a_{i,i-1}  \\
   \end{array}
  \right|}
  {a_{j,j}\cdots a_{i,i}}.
 \end{align*}
 Here we understand $|\,\emptyset\,|=1$.
\end{lem}
\begin{proof}
 We can obtain this lemma by cofactor expansions.  
\end{proof}
\begin{cor} \label{cor:sumis0}
 The sum of the components in each row of matrices $A_1^{-1}$ and $A_2^{-1}$ is 0 except for the first row.
\end{cor}
\begin{proof}
By definitions, we have
 \begin{align*}
  A_1
  =\left(
 \begin{array}{cccccc}
  1 & 0& 0&\cdots&\cdots&0  \\
  1 & a_{2,2} &0&0&\cdots&0 \\
  1 & a_{4,2} & a_{4,3} &0&\cdots&0 \\
  \vdots &&\ddots &\ddots&\ddots&\vdots \\  
  \vdots &&&\ddots&\ddots&0 \\  
  1 & a_{N',2}& \cdots &\cdots& a_{N',N'-1} & a_{N',N'} \\
 \end{array}
 \right).
 \end{align*}
Put $\boldsymbol{1}:=(\underbrace{1,\dots,1}_{N'})^\mathsf{T}$ 
and $\boldsymbol{0}':=(1,\underbrace{0,\dots,0}_{N'-1})^\mathsf{T}$.
From Proposition \ref{prop:lowerMat}, the $k$-th row of $A_1^{-1} \boldsymbol{1}$ is written by
 \begin{align*}
  \frac{
  \left|
   \begin{array}{cccccccc}
    1 & a_{2,2} &0&0&\cdots&0 \\
    1 & a_{3,2} & a_{3,3} &0&\cdots&0 \\
    \vdots &\vdots &\vdots& \ddots & \ddots&\vdots \\  
    \vdots &\vdots &\vdots&& \ddots & 0 \\  
    1 & a_{k-1,2} & \cdots & \cdots & \cdots & a_{k-1,k-1}  \\
    1 & a_{k,2} & \cdots & \cdots & \cdots & a_{k,k-1}  \\
   \end{array}
  \right|}
  {a_{2,2}\cdots a_{k,k}}
  -
  \sum_{j=2}^{k} (-1)^{j}
  \frac{
  \left|
   \begin{array}{cccccccc}
    a_{j+1,j} & a_{j+1,j+1} &0&\cdots&0 \\
    \vdots &\vdots& \ddots & \ddots&\vdots \\  
    \vdots &\vdots&& \ddots & 0 \\  
    a_{k-1,j} & \cdots & \cdots & \cdots & a_{k-1,k-1}  \\
    a_{k,j} & \cdots & \cdots & \cdots & a_{k,k-1}  \\
   \end{array}
  \right|}
  {a_{j,j}\cdots a_{k,k}}.
 \end{align*}
 Then, by Lemma \ref{lem:simplefact}, we have $A_1^{-1} \boldsymbol{1}=\boldsymbol{0}'$. 
 Similarly, we have $A_2^{-1} \boldsymbol{1}=\boldsymbol{0}'$. 
Hence we find the result.  
\end{proof}

\section{Analytic properties of $\zeta (-n,s+n)$}
To show Theorem \ref{th:main1}, we determine the location of the poles and their residues of $\zeta (-n,s+n)$ in this section. 
Note that the meromorphic continuation and the possible singularities of $\zeta (s_1, s_2)$ with $s_1, s_2 \in {\mathbb{C}}$ are already given in Akiyama Egami and Tanigawa \cite{AET01}, Matsumoto \cite{Mat02} and Zhao \cite{Zha00}. 
For reader's convenience, we give some details of \cite[Section 4]{Mat02} which treats the analytic continuation of $\zeta (s_1, s_2)$.

According to \cite[(4.4)]{Mat02}, we have
\begin{align} \label{eq:AETM1}
\begin{split}
\zeta (s_1, s_2) = & \, \frac{\zeta (s_1+s_2-1)}{s_2-1} 
+ \sum_{k=0}^{M-1} \binom{-s_2}{k} \zeta (-k) \zeta (s_1+s_2+k)
\\ &+ \frac{1}{2 \pi i} \int_{(M-\varepsilon)} \frac{\Gamma (s_2+z) \Gamma (-z)}{\Gamma (s_2)} 
\zeta (s_1+s_2+z) \zeta (-z) dz,
\end{split}
\end{align} 
where $M$ is a positive integer, $\varepsilon$ is a positive number, and the path of integration is the vertical line from $M-\varepsilon-i\infty$ to $M-\varepsilon+i\infty$. 
Note that the integral above can be continued holomorphically to the region
 \begin{align*}
 \bigl\{ (s_1, s_2 ) \in {\mathbb{C}^2} \mid \Re (s_2) > -M + \varepsilon, \,\,\, \Re (s_1+s_2) > 1-M + \varepsilon \bigr\}
 \end{align*}
since in this region the poles of the integrand are not on the path of integration. Thus (\ref{eq:AETM1}) gives the meromorphic continuation of $\zeta (s_1, s_2)$ to whole ${\mathbb{C}}^2$ which is holomorphic in
\[
\bigl\{ (s_1, s_2) \in {\mathbb{C}}^2 \mid s_2 \ne 1, \,\,\, s_1 +s_2 \not \in \{ 2,1, 0, -2,-4,-6, \ldots \} \bigr\}
\]
because we can take $M$ arbitrarily and one has $\zeta (-2l)=0$, where $l$ is a positive integer (see also \cite[p.~109]{AET01}). 

By using the formula (\ref{eq:AETM1}), we show the following. 
\begin{thm} \label{th:pole}
 The function $\zeta (0,s)$ is analytic except for $s=2,1$ 
 and has a pole at $s=2$ whose residue is $1$ and a pole at $s=1$ whose residue is $2\zeta (0)$. 
 The function $\zeta (-1,s+1)$ is analytic except for $s=2,1$ 
 and has a pole at $s=2$ whose residue is $1/2$ and a pole at $s=1$ whose residue is $\zeta (0)$. 

 For an integer $n \ge2$, the function $\zeta (-n,s+n)$ is analytic 
 except for $s=2,1, 0, -2, -4$, $\dots,-2\lfloor n/2 \rfloor +2$ and has a poles
 \begin{align} \label{eq:pole1}
  \begin{array}{cccc}
   \mbox{at} & s=2, & \quad \mbox{ residue} & {\displaystyle{\frac{1}{n+1}}}, \\
   \mbox{at} & s=1, & \quad \mbox{ residue} & {\displaystyle{\zeta (0) = -\frac{1}{2}}},\\
   \mbox{at} & s=-2k, & \quad \mbox{ residue} & {\displaystyle{\binom{2k-n}{2k+1} \zeta (-2k-1)}},
  \end{array}
 \end{align} 
 where $k = 0, 1, 2,\dots, \lfloor n/2 \rfloor -1$.
\end{thm}
\begin{proof}
Put $s_1 = -n$ and $s_2 = s+n$, where $n$ is a non-negative integer. Then one has
\begin{align} \label{eq:AETM2}
 \begin{split}
  \zeta (-n, s+n) 
  &=\frac{\zeta (s-1)}{s+n-1} + \sum_{k=0}^{M-1} \binom{-s-n}{k} \zeta (-k) \zeta (s+k) \\ 
  &\quad +\frac{1}{2 \pi i} \int_{(M-\varepsilon)} \frac{\Gamma (s+n+z) \Gamma (-z)}{\Gamma (s+n)} 
   \zeta (s+z) \zeta (-z) dz.
 \end{split}
\end{align} 
When $n=0$, the function $\zeta (0,s)$ has a pole at $s=2$ whose residue is $1$ and a pole at $s=1$ whose residue is $\zeta (0)$ by the function $\zeta (s-1)/(s-1)$, and a pole at $s=1$ whose residue is $\zeta (0)$ by $\binom{-s}{0} \zeta (0) \zeta (s)$. There is no pole in the sum
 \begin{align*}
  \sum_{k=1}^{M-1} \binom{-s}{k} \zeta (-k) \zeta (s+k)
 \end{align*}
since the zero of $\binom{-s}{k}$ and the pole of $\zeta (s+k)$ cancel each other when $k \ge 1$. 

Similarly, when $n=1$, the function $\zeta (-1,s+1)$ has a pole at $s=2$ whose residue is $1/2$ by the function $\zeta (s-1)/s$ and a pole at $s=1$ whose residue is $\zeta (0)$ by $\binom{-s-1}{0} \zeta (0) \zeta (s)$. At $s=0$, we have the pole whose residue is $\zeta(-1)$ in the function $\zeta (s-1)/s$ but this is canceled by the pole in the function $\binom{-s-1}{1} \zeta (-1) \zeta (s+1)$. There is no pole in the sum
 \begin{align*}
 \sum_{k=2}^{M-1} \binom{-s-1}{k} \zeta (-k) \zeta (s+k)
 \end{align*}
since the zero of $\binom{-s-1}{k}$ and the pole of $\zeta (s+k)$ cancel each other when $k \ge2$. 

Assume that $m$ is a positive integer and consider the poles of $\zeta (-2m,s+2m)$. 
The function 
\[
 \frac{\zeta (s-1)}{s+2m-1}
\]
has only one pole at $s=2$ whose residue is $(2m+1)^{-1}$. 
Note that the pole of $(s+2m-1)^{-1}$ at $s= -2m+1$ is canceled by $\zeta (s-1)$ with $s-1=-2m$ by the fact that $\zeta (-2m)=0$. 
Next we consider the function
\[
 \sum_{k=0}^{M-1} f_k (s) \zeta (s+k), \qquad f_k (s) := \binom{-s-2m}{k} \zeta (-k).
\]
The function $f_0(s) \zeta (s)$ has a pole at $s=1$ whose residue is $\zeta (0)$. 
Let $l$ be a positive integer. 
Then the function
\[
 f_{2l-1} (s) \zeta (s+2l-1)
\]
has a pole at $s+2l-1=1$ whose residue is $f_{2l-1} (2-2l)$ if $l \le m$. 
When $l > m$, the function $f_{2l-1} (s) \zeta (s+2l-1)$ does not have a pole since the function $f_{2l-1} (s)$ has a zero at $s=2-2l$ by
\[
 f_{2l-1} (s)
 =\frac{(-s-2m) (-s-2m-1) \cdots (-s -2l+2) \cdots (-s -2m-2l-2)}{(2l-1)!} \zeta (1-2l).
\] 
On the other hand, the function
\[
 f_{2l} (s) \zeta (s+2l)
\]
is holomorphic since $\zeta (-2l)=0$ if $l$ is a positive integer. 

Similarly, we consider the poles of $\zeta (-2m-1,s+2m+1)$. 
The function 
\[
 \frac{\zeta (s-1)}{s+2m}
\]
has a pole at $s=2$ whose residue is $(2m+2)^{-1}$ and a pole at $s=-2m$ whose residue is $\zeta (-2m-1)$. Now we pick up the pole of the function
\[
\sum_{k=0}^{M-1} g_k (s) \zeta (s+k), \qquad g_k(s) := \binom{-s-2m-1}{k} \zeta (-k).
\]
The function $g_0(s) \zeta (s)$ has a pole at $s=1$ whose residue is $\zeta (0)$. When $l \le m$, the function 
\[
g_{2l-1} (s) \zeta (s+2l-1)
\]
has a pole at $s=2-2l$ whose residue is $g_{2l-1} (2-2l)$. When $l=m+1$ the function $g_{2m+1} (s) \zeta (s+2m+1)$ has a simple pole at $s=-2m$ whose residue is
\[
g_{2m+1} (2m) = \binom{2m - 2m-1}{2m+1} \zeta (-2m-1) = -\zeta (-2m-1).
\]
It should be emphasised that the pole above is canceled by the pole at $s=-2m$ which comes from the function $(s+2m)^{-1} \zeta (s-1)$. If $l > m+1$ the function $g_{2l-1} (s) \zeta (s+2l-1)$ does not have a pole since the function $f_{2l-1} (s)$ has a zero at $s=2-2l$. 
In addition, the function
\[
g_{2l} (s) \zeta (s+2l)
\]
is holomorphic since $\zeta (-2l)=0$ if $l$ is a positive integer. Therefore, we have the result.
\end{proof}

\section{Proofs of the last part of Theorem \ref{th:main1}{} and Theorem \ref{th:main2}}
In order to show Theorems \ref{th:main1} and \ref{th:main2}, we prove some corollaries below deduced form Theorem \ref{th:pole}. 
\begin{cor} \label{cor:2m}
 Let $m$ be a positive integer. 
 Then there do not exist complex numbers $c_0,c_1 , c_2$, $\dots, c_{2m-1}$ such that
 \begin{align*}
  0 &\equiv \zeta (-2m,s+2m) + c_{2m-1} \zeta (-2m+1,s+2m-1) + c_{2m-2} \zeta (-2m+2,s+2m-2) \\
  &\quad +\cdots + c_{2} \zeta (-2,s+2) + c_{1} \zeta (-1,s+1) + c_{0} \zeta (0,s)/2.
 \end{align*}
\end{cor}
\begin{proof}
 From Theorem \ref{th:pole}, the function $\zeta (-2m,s+2m)$ has a pole at $s=2-2m$. 
 On the contrary, the other functions $\zeta (-2m+1,s+2m-1)$, $\ldots$, $\zeta (0,s)$ do not have a pole there. 
\end{proof}

\begin{cor} \label{cor:2m+1cofzero}
 Let $m$ be a positive integer and suppose that there exist complex numbers $c_0, c_1 ,\dots, c_{2m+1}$ such that
 \begin{align} \label{eq:uniorno1}
  \begin{split}
   0 &\equiv c_{2m+1} \zeta (-2m-1,s+2m+1) + \cdots + c_{1} \zeta (-1,s+1) + c_{0} \zeta (0,s)/2 .
  \end{split}
 \end{align} 
 Then we have $c_0 + c_1 + c_2 + \cdots + c_{2m+1} =0$.
\end{cor}
\begin{proof}
 The functions $\zeta (0,s)/2$ and $\zeta (-l,s+l) \; (1 \le l \le 2m+1)$ have a pole at $s=1$ whose residue is $\zeta (0)$ by Theorem \ref{th:pole}. 
 Thus we have $c_0 + c_1 + \cdots + c_{2m+1} =0$ since the constant function $0$ is analytic. 
\end{proof}

\begin{cor} \label{lem:2m+1}
 Let $m$ be a non-negative integer and suppose that
 \begin{align} \label{eq:uniorno1}
  \begin{split}
   0 &\equiv \zeta (-2m-1,s+2m+1) + c_{2m} \zeta (-2m,s+2m) \\
   &\quad +c_{2m-2} \zeta (-2m+2,s+2m-2) + \cdots + c_{2} \zeta (-2,s+2)  + c_{0} \zeta (0,s)/2 .
  \end{split}
 \end{align} 
 Then the coefficients $c_0, c_2, c_4, \ldots , c_{2m}$ are uniquely determined. 
\end{cor}

\begin{proof}
 By the poles of $\zeta (-2m-1,s+2m+1)$ and $\zeta (-2m,s+2m)$ at $s= 2m-2$ 
 and their residue shown in Theorem \ref{th:pole}, the  coefficient $c_{2m}$ which satisfies
 \[
  \binom{-3}{2m-1} + c_{2m} \binom{-2}{2m-1} =0
 \] 
 is determined uniquely, namely, one has
 \begin{align} \label{eq:c}
  c_{2m} = - \frac{2m+1}{2}.
 \end{align}
 From the poles of $\zeta (-2m-1,s+2m+1)$, $\zeta (-2m,s+2m)$, $\zeta (-2m+1,s+2m-1)$, 
 and $\zeta (-2m+2,s+2m-2)$ at $s= 2m-4$ and their residue, we have
 \[
  \binom{-5}{2m-3} + c_{2m} \binom{-4}{2m-3} + 0 \binom{-3}{2m-3} + c_{2m-2} \binom{-2}{2m-3} = 0. 
 \] 
 Hence, the coefficient $c_{2m-2}$ is determined uniquely since one has $c_{2m} = - (2m+1)/2$. 
 Similarly it holds that 
 \begin{align*} 
  \begin{split}
   &\binom{-7}{2m-5} + c_{2m} \binom{-6}{2m-5} + 0 \binom{-5}{2m-5} + c_{2m-2} \binom{-4}{2m-5} \\  
   & +0 \binom{-3}{2m-5} + c_{2m-4} \binom{-2}{2m-5} =0.
  \end{split}
 \end{align*} 
 Thus we can uniquely determine $c_{2m-4}$, $c_{2m-6}, \ldots , c_2$, inductively. 
 The coefficient $c_0$ is determined uniquely by Corollary \ref{cor:2m+1cofzero}. 
\end{proof}

Now we prove the remaining parts of our main theorems given in Section 2.
\begin{proof}[Proof of the last part of Theorem \ref{th:main1}]
 From Corollary \ref{cor:2m}, 
 it suffices to show that all equations of the form \eqref{eq:uniorno1} can be derived from \eqref{eq:eq1}, 
 which is readily apparent from Corollary \ref{lem:2m+1} and \eqref{eq:eq1}. 

 In fact, from the equation \eqref{eq:eq1}, we have
 \begin{align} \label{eq:d}
 \begin{split}
  &(\zeta(-1,s+1),\zeta(-3,s+3),\dots, \zeta(-2N'+1,s+2N'-1) )^\mathsf{T} \\
  &=A_2 A_1^{-1} (\zeta(0,s)/2,\zeta(-2,s+2),\zeta(-4,s+4),\dots, \zeta(-2N'+2,s+2N'-2) )^\mathsf{T}, 
 \end{split}
 \end{align}
 where $N'$ is defined in Section 1. 
 Since $A_2 A_1^{-1}$ is a lower triangular matrix, we have 
 \begin{equation} \label{eq:mrela1}
 \begin{split}
  \zeta (-2m-1,s+2m+1) 
  &\equiv 
  c_{2m} \zeta (-2m,s+2m) +c_{2m-2} \zeta (-2m+2,s+2m-2)  \\
  &\quad + \cdots + c_{2} \zeta (-2,s+2)  + c_{0} \zeta (0,s)/2
 \end{split}
 \end{equation} 
 for $0 \le m\le N'$,
 where $c_0, \ldots , c_{2m}$ are rational numbers. 
 From Corollary \ref{lem:2m+1}, there are no other ${\mathbb{Q}}$-linear relations 
 among $\zeta (-2m-1,s+2m+1)$ and $\zeta (-2m,s+2m), \zeta (-2m+2,s+2m-2), \ldots, \zeta (0,s)$. 
 Therefore, the equation \eqref{eq:eq1} gives all $\mathbb{Q}$-linear relations of $\mathcal{Z}_{N}$.
\end{proof}

\begin{proof}[Proof of Theorem \ref{th:main2}]
 When $N=1$, by Lemma \ref{lem:main1}, 
 we easily see that $\mathcal{Z}_{1}=\mathbb{Q} \cdot \zeta(-1,s+1) =\mathbb{Q} \cdot \zeta(0,s)$. 
 If $N=2$, we have $\mathcal{Z}_{2}= \bigoplus_{c=0,1} \mathbb{Q} \cdot \zeta(2c,s)$ by Corollary \ref{cor:2m}. 
 For $N=3$, one has $\mathcal{Z}_{3}= \mathcal{Z}_{2}$ by the $\mathbb{Q}$-linear relation \eqref{eq:eq1} or \eqref{eq:mrela1}. 
 Similarly, we have $\mathcal{Z}_{4}= \bigoplus_{c=0}^{2} \mathbb{Q} \cdot \zeta(2c,s)$ and $\mathcal{Z}_{5}= \mathcal{Z}_{4}$ 
 from Corollary \ref{cor:2m} and Theorem \ref{th:main1} or (\ref{eq:mrela1}), respectively.  
 Hence, we obtain Theorem \ref{th:main2} by Corollary \ref{cor:2m} and Theorem \ref{th:main1}, (\ref{eq:eq1}) or (\ref{eq:mrela1}), inductively. 
\end{proof}

\section{Examples}
This section describes concrete examples when $N=12$ in Section 2.
\begin{ex}[Matrices] \label{acd}
 The specific values of the matrices $A,A_1,A_2,A_1^{-1},A_2^{-1}$ are shown below.
\begin{align*}
A=
\left(
\begin{array}{cccccc}
 1 & 0 & 0 & 0 & 0 & 0 \\
 1 & 0 & 0 & 0 & 0 & 0 \\
 1 & -2 & 0 & 0 & 0 & 0 \\
 1 & -3 & 0 & 0 & 0 & 0 \\
 1 & -4 & 2 & 0 & 0 & 0 \\
 1 & -5 & 5 & 0 & 0 & 0 \\
 1 & -6 & 9 & -2 & 0 & 0 \\
 1 & -7 & 14 & -7 & 0 & 0 \\
 1 & -8 & 20 & -16 & 2 & 0 \\
 1 & -9 & 27 & -30 & 9 & 0 \\
 1 & -10 & 35 & -50 & 25 & -2 \\
 1 & -11 & 44 & -77 & 55 & -11 \\
\end{array}
\right),
\end{align*} 
\begin{align*}
A_1=
\left(
\begin{array}{cccccc}
 1 & 0 & 0 & 0 & 0 & 0 \\
 1 & -2 & 0 & 0 & 0 & 0 \\
 1 & -4 & 2 & 0 & 0 & 0 \\
 1 & -6 & 9 & -2 & 0 & 0 \\
 1 & -8 & 20 & -16 & 2 & 0 \\
 1 & -10 & 35 & -50 & 25 & -2 \\
\end{array}
\right),
\; A_2=
\left(
\begin{array}{cccccc}
 1 & 0 & 0 & 0 & 0 & 0 \\
 1 & -3 & 0 & 0 & 0 & 0 \\
 1 & -5 & 5 & 0 & 0 & 0 \\
 1 & -7 & 14 & -7 & 0 & 0 \\
 1 & -9 & 27 & -30 & 9 & 0 \\
 1 & -11 & 44 & -77 & 55 & -11 \\
\end{array}
\right).
\end{align*}

By using Proposition \ref{prop:lowerMat}, we can calculate 
\begin{align*}
A_1^{-1}&=
\left(
\begin{array}{cccccc}
 1 & 0 & 0 & 0 & 0 & 0 \\
 1/2 & -1/2 & 0 & 0 & 0 & 0 \\
 1/2 & -1 & 1/2 & 0 & 0 & 0 \\
 5/4 & -3 & 9/4 & -1/2 & 0 & 0 \\
 13/2 & -16 & 13 & -4 & 1/2 & 0 \\
 227/4 & -140 & 115 & -75/2 & 25/4 & -1/2 \\
\end{array}
\right),\\
\; A_2^{-1}&=
\left(
\begin{array}{cccccc}
 1 & 0 & 0 & 0 & 0 & 0 \\
 1/3 & -1/3 & 0 & 0 & 0 & 0 \\
 2/15 & -1/3 & 1/5 & 0 & 0 & 0 \\
 8/105 & -1/3 & 2/5 & -1/7 & 0 & 0 \\
 8/105 & -4/9 & 11/15 & -10/21 & 1/9 & 0 \\
 32/231 & -8/9 & 5/3 & -29/21 & 5/9 & -1/11 \\
\end{array}
\right).
\end{align*}
Note that the sum of the components in each row of matrices $A_1^{-1}$ and $A_2^{-1}$ is 0 except for the first row (see Corollary \ref{cor:sumis0}). 
\end{ex}

\begin{ex}[Relations between the Euler-Zagier and Tornheim double zeta functions] \label{ex:zt}
 By Proposition \ref{prop:Tornheim}, we have
 \begin{align*}
 \left(
 \begin{array}{c}
 \zeta(0,s)/2 \\
 \zeta(-1,s+1) \\
 \zeta(-2,s+2) \\
 \zeta(-3,s+3) \\
 \cdots \\
 \zeta(-10,s+10) \\
 \zeta(-11,s+11) \\
 \end{array} 
 \right) 
 =\frac{A}{2}\;
 \left(
 \begin{array}{c}
 T(0,0;s) \\
 T(0,0;s) \\
 T(-1,-1;s+2) \\
 T(-1,-1;s+2) \\
 \cdots \\
 T(-5,-5;s+10) \\
 T(-5,-5;s+10) \\
 \end{array}
 \right),
 \end{align*}
 i.e.,
 \begin{align*}
 \left(
 \begin{array}{c}
 \zeta(0,s)/2 \\
 \zeta(-2,s+2) \\
 \zeta(-4,s+4) \\
 \zeta(-6,s+6) \\
 \zeta(-8,s+8) \\
 \zeta(-10,s+10) \\
 \end{array} 
 \right) 
 =\frac{A_1}{2}\;
 \left(
 \begin{array}{c}
 T(0,0;s) \\
 T(-1,-1;s+2) \\
 T(-2,-2;s+4) \\
 T(-3,-3;s+6) \\
 T(-4,-4;s+8) \\
 T(-5,-5;s+10) \\
 \end{array}
 \right),
 \end{align*}
 \begin{align*}
 \left(
 \begin{array}{c}
 \zeta(-1,s+1) \\
 \zeta(-3,s+3) \\
 \zeta(-5,s+5) \\
 \zeta(-7,s+6) \\
 \zeta(-9,s+9) \\
 \zeta(-11,s+11) \\
 \end{array} 
 \right) 
 =\frac{A_2}{2}\;
 \left(
 \begin{array}{c}
 T(0,0;s) \\
 T(-1,-1;s+2) \\
 T(-2,-2;s+4) \\
 T(-3,-3;s+6) \\
 T(-4,-4;s+8) \\
 T(-5,-5;s+10) \\
 \end{array}
 \right).
 \end{align*}
 From the equations above, we have
 \begin{align}  \label{eq:relmat}
 \left(
 \begin{array}{c}
 T(0,0;s) \\
 T(-1,-1;s+2) \\
 T(-2,-2;s+4) \\
 T(-3,-3;s+6) \\
 T(-4,-4;s+8) \\
 T(-5,-5;s+10) \\
 \end{array}
 \right)
 =2A_1^{-1} \;
  \left(
 \begin{array}{c}
 \zeta(0,s)/2 \\
 \zeta(-2,s+2) \\
 \zeta(-4,s+4) \\
 \zeta(-6,s+6) \\
 \zeta(-8,s+8) \\
 \zeta(-10,s+10) \\
 \end{array} 
 \right) 
 =2A_2^{-1} \;
  \left(
 \begin{array}{c}
 \zeta(-1,s+1) \\
 \zeta(-3,s+3) \\
 \zeta(-5,s+5) \\
 \zeta(-7,s+7) \\
 \zeta(-9,s+9) \\
 \zeta(-11,s+11) \\
 \end{array} 
 \right).  
 \end{align}
\end{ex}

\begin{ex}[Basis representation] 
 By using \eqref{eq:relmat} or \eqref{eq:d} with $N=12$, we have 
 \begin{align*}
  \zeta (-1,s+1) &= \zeta (0,s)/2 \\
  \zeta(-3, s + 3) 
  &= 3 \zeta(-2, s + 2)/2 -\zeta(0, s )/4, \\
  \zeta(-5, s + 5) 
  &= 5 \zeta(-4, s + 4)/2 -5 \zeta(-2, s + 2)/2 + \zeta(0, s )/2, \\
  \zeta(-7, s + 7) 
  &= 7 \zeta(-6, s + 6)/2 -35 \zeta(-4, s + 4)/4 + 21 \zeta(-2, s + 2)/2 - 17 \zeta(0, s )/8, \\
  \zeta(-9, s + 9) 
  &= 9 \zeta(-8, s + 8)/2 - 21 \zeta(-6, s + 6) +63 \zeta(-4, s + 4) - 153 \zeta(-2, s + 2)/2 \\
  &\quad + 31/2 \zeta(0, s ), \\ 
  \zeta(-11, s + 11) 
  &= 11 \zeta(-10, s + 10)/2 - 165 \zeta(-8, s + 8)/4 +231 \zeta(-6, s + 6) \\
  &\quad - 2805 \zeta(-4, s + 4)/4 +1705 \zeta(-2, s + 2)/2 - 691 \zeta(0, s )/4. 
 \end{align*}
 Similar to the above, we can also obtain
 \begin{align*}
  \zeta(-13, s + 13) 
  &= 13 \zeta(-12, s + 12)/2 - 143 \zeta(-10, s + 10)/2 +1287 \zeta(-8, s + 8)/2 \\
  &\quad - 7293 \zeta(-6, s + 6)/2 + 22165 \zeta(-4, s + 4)/2 - 26949 \zeta(-2, s + 2)/2 \\
  &\quad + 5461 \zeta(0, s )/2, \\
  \zeta(-15, 15 + s) 
  &= 15 \zeta(-14, 14 + s)/2 - 455 \zeta(-12, s + 12)/4 + 3003 \zeta(-10, s + 10)/2 \\
  &\quad - 109395 \zeta(-8, s + 8)/8 + 155155 \zeta(-6, s + 6)/2 - 943215 \zeta(-4, s + 4)/4 \\
  &\quad + 573405 \zeta(-2, s + 2)/2 - 929569 \zeta(0, s )/16.
 \end{align*} 
\end{ex}
\begin{rem}
 By the method given in the proof of Corollary \ref{lem:2m+1}, we can determine the coefficients of $\mathbb{Q}$-linear relations above (see \eqref{eq:c}).
 In the sixth and eighth relations, we can find prime numbers $691$ and $3617$ derived from $929569 = 257 \times 3617$. 
 These are caused by $\zeta (-11)$ and $\zeta (-15)$, respectively (see Theorem \ref{th:pole} and Corollary \ref{cor:2m+1cofzero}). 
\end{rem}


\end{document}